\documentclass[a4,11pt]{amsart}%
\usepackage{amsmath,amssymb,amsfonts,enumerate,amsthm, amscd,}
\usepackage[all]{xy}
\usepackage{graphicx}
\usepackage{amsmath}
\usepackage{amsfonts}
\usepackage{amssymb}
\usepackage{hyperref}%
\providecommand{\U}[1]{\protect\rule{.1in}{.1in}}
\newtheorem{thm}{Theorem}[section]
\newtheorem{cor}[thm]{Corollary}
\newtheorem{lem}[thm]{Lemma}
\newtheorem{prop}[thm]{Proposition}

\newtheorem{exam}[thm]{Example}

\theoremstyle{definition}

\theoremstyle{remark}

\theoremstyle{Definition and Notation}

\begin{document}

\title{ Commutative rings with $n$-$1$-absorbing prime factorization}

\author[A. El Khalfi]{Abdelhaq El Khalfi}
\address{(Abdelhaq El Khalfi) Laboratory of Mathematical Analysis, Algebra and Applications., Faculty of Sciences Ain Chock, Hassan II University of Casablanca, Morocco.
$$E-mail: abdelhaqelkhalfi@gmail.com$$}

\author[H. Laarabi]{Hicham Laarabi}
\address{(Hicham Laarabi) Laboratory of Data Engineering and Intelligent Systems.,
Faculty of Sciences Ain Chock, Hassan II University of Casablanca, Morocco.
$$E-mail\ address:\ hichamlaarabi036@gmail.com$$}

\author[S. Ko\c{c}]{Suat Ko\c{c}}
\address{(Suat Ko\c{c})
Department of Mathematics, Marmara University, Istanbul, Turkey.
$$E-mail\ address:\ suat.koc@marmara.edu.tr$$}


%Subject classification numbers        13B99           13B21

\begin{abstract}
Let $R$ be a commutative ring with $1\neq 0$ and $n$ be a fixed positive integer. A proper ideal $I$ of $R$ is said to be an \textit{$n$-OA ideal} if whenever $a_1a_2\cdots a_{n+1}\in I$ for some nonunits $a_1,a_2,\ldots,a_{n+1}\in R$, then $a_1a_2\cdots a_n\in I$ or $a_{n+1}\in I$. A commutative ring $R$ is said to be an \textit{$n$-OAF ring} if every proper ideal $I$ of $R$ is a product of finitely many $n$-OA ideals. In fact, $1$-OAF rings and $2$-OAF $2$-OAF-rings are exactly the general ZPI rings and OAF rings, respectively. In addition to giving various properties of $n$-OAF rings, we give a characterization of Noetherian von Neumann regular rings in terms of our new concept. Furthermore, we investigate the $n$-OAF property of some extension of rings such as the polynomial ring $R[X]$, the formal power series ring $R[[X]]$, the ring of $A+XB[X]$, and the trivial extension $R=A\propto E$ of an $A$-module $E$.
 
\end{abstract}

\maketitle

%%%%%%%%%%%%%%%%%%%%%%%%%%%%%%%%%%%%%%%%%%%%%%%%%%%%%%%%%

\bigskip
 %%%%%%%%%%%%%%%%%%%%%%%%%%%%%%%%%%%%%%

%%%%%%%%%%%%%%%%%%%%%%%%%%%%%%%%%%%%%%%%%%%%%%%%%%%%%%%%
%%%%%%%%%%%%%%%%%%%%%%%%%%%%%%%%%%%%%%%%%%%%%%%%%%%%%%%%
%%%%%%%%%%%%%%%%%%%%%%%%%%%%%%%%%%%%%%%%%%%%%%%%%%%%%%%%

%%%%%%%%%%%%%%%%%%%%%%%%%%%%%%%%%%%%%%%%%%%%%%%%%%%%%%%%%
%%%INTRODUCTION%%%%%%%%%%%%%%%%%%%%%%%%%%%%%%%%%%%%%%%%%%

\section{Introduction}
This paper considers only commutative rings with a nonzero identity. The symbol $R$ always denotes such a ring. The sets $Spec(R)$, $Max(R)$, and $Min(R)$ represent the prime ideals, maximal ideals, and minimal prime ideals of $R$, respectively. The notation $Nil(R)$ refers to the set of all nilpotent elements of $R$. For an ideal $I$ of $R$ and $x \in R$, the residual of $I$ by $x$ is defined as $(I:x) = \{ r \in R \mid rx \in I \}$. The concept of prime ideals and its generalisations have a distinguished place in commutative algebra. Many important classes of rings, such as Dedekind domains (general ZPI-rings), Noetherian rings, Q-rings, valuation domains, UFDs, PIDs, can be defined or characterised in terms of prime ideals or their generalisations, and they also have some applications to other areas such as General topology, Algebraic geometry, Cryptology, Graph theory, etc. In 2009, Badawi in his celebrated paper \cite{B} defined the concept of 2-absorbing ideal which is a generalization of prime ideal as follows: a proper ideal $I$ of $R$ is said to be a \textit{2-absorbing ideal} (for short, \textit{TA-ideal}) if whenever $abc\in I$ for some $a,b,c\in R$, then either $ab\in I$ or $ac\in I$ or $bc\in I$. Afterwards, Anderson and Badawi generalised this concept to $n$-absorbing ideals in \cite{AB2}. Let $n$ be a fixed positive integer. A proper ideal $I$ of $R$ is said to be an \textit{$n$-absorbing ideal} if whenever $a_1a_2\cdots a_na_{n+1}\in I$ for some $a_1,a_2,\ldots,a_{n+1}\in R$, then $I$ contains a product of $n$ elements in $a_1,a_2,\ldots,a_{n+1}$. It is well known that every prime ideal is also a 2-absorbing ideal. However, the converse is not true in general. For instance, in $k[X,Y]$, where $k$ is a field, $I=(X,Y)$ is a 2-absorbing ideal which is not prime.\\

In \cite{ABS2}, Mukhtar et al. used the concept of 2-absorbing ideals to study TA-factorisation in commutative rings. By a $TA$-factorization of an ideal $I$ of $R$, we mean an expression of $I$ as a product $\prod_{i=1}^rI_i$ of  $TA$-ideals of $R$. A ring $R$ is said to be a \textit{TAF-ring} if its every proper ideal $I$ of $R$ has a TA-factorisation. In \cite[Theorem 3.3]{ABS2}, the authors showed that every TAF-ring is a finite direct product of one-dimensional domains and zero-dimensional local rings having a nilpotent maximal ideal. Afterwards, in \cite{TAF},  Ahmed et al. studied commutative rings whose proper ideals have an $n$-absorbing factorisation, that is, each proper ideal $I$ of $R$ can be written as a product of finitely many $n$-absorbing ideals of $R$. The authors called \textit{AF-$dim(R)$} the minimum positive integer $n$ such that every ideal of $R$ has an $n$-absorbing factorisation. If no such $n$ exists, set AF-$dim(R)=\infty$. Also, a \textit{$FAF$-ring} is a ring such that AF-$dim(R) < \infty$. Therefore, $AF$-$dim(R)$ measures in some sense how far $R$ is from being a general ZPI-ring (a ring whose proper ideals can be written as a product of prime ideals). 
In \cite{YNN}, Yassine et al. introduced the concept of a $1$-absorbing prime ideal (for short, $OA$-ideal), which is an intermediate class of ideals between 2-absorbing ideals and prime ideals. A proper ideal $I$ of $R$ is called an \textit{$OA$-ideal} if $abc \in I$ for some nonunits $a,b,c\in R$, then $ab\in I$  or $c\in I$. In this sense, in \cite{ABS1}, El Khalfi et al. studied commutative rings whose ideals have an $ OA$-factorisation (every proper ideal of $R$ is the product of a finite many $OA$-ideals), and they called such rings $OAF$-rings. In particular, the authors  proved that $OAF$-rings are exactly the rings whose Krull dimension $dim(R)$ is at most 1 and each proper principal ideal has $ OA$-factorisation (See, \cite[Theorem 4.2]{ABS1}.\\

In very recent studies \cite{BEM} and \cite{gulak}, the authors have independently introduced and studied the concept $n$-$1$-absorbing prime ideals (our abbreviation, $n$-OA-ideal), which is an intermediate class between prime ideals and $n$-absorbing ideals. Let $n$ be a fixed positive integer. A proper ideal $I$ of $R$ is called an \textit{$n$-OA ideal} if whenever $a_1a_2\cdots a_{n+1} \in I$ for some nonunits $a_1,a_2,\ldots,a_{n+1} \in R$, then $a_1\cdots a_n\in I$ or $a_{n+1}\in I$. Motivated by the studies \cite{TAF}, \cite{ABS1} and \cite{ABS2}, in this paper,  we investigate and study the commutative rings whose every ideal has an $n$-OA factorisation, that is, every ideal is a product of finitely many $n$-OA ideals, such rings are called $n$-OAF rings. Since $1$-OA ideals and prime ideals ($2$-OA ideals and 1-absorbing prime ideals) coincide, it is easy to see that $1$-OAF rings ($2$-OAF rings) are exactly general ZPI-rings (OAF-rings). Thus, $n$-OAF rings are a natural generalisation of general ZPI-rings and $OAF$-rings. Among the other results in Section 2, we show that in an $n$-OAF ring $R$, the set $Min(I)$ of all prime ideals that are minimal over an ideal $I$ is always finite (See, Proposition \ref{promin}). Also, we prove that in a local $n$-OAF-ring $R$, each nonmaximal minimal prime ideal is principal (See, Lemma \ref{lm2}). Furthermore, every $n$-OAF ring is either a zero-dimensional ring or a 1-dimensional domain (See, Theorem \ref{thm12}). Also, we study the stability of $n$-OA factorisation in the factor ring, in localisation and in the direct product of rings (See, Proposition \ref{prp5}, Proposition \ref{ploc} and Proposition \ref{car}). In Section 3, we deal with the $n$-OAF properties of polynomial rings $R[X]$, formal power series ring $R[[X]]$ and the ring $A+XB[X]$, where $A\subseteq B$ is a ring extension (See, Proposition \ref{poly}, Theorem \ref{series} and Proposition \ref{extension}). Combining Proposition \ref{poly} and Theorem \ref{series}, we prove that $R[X]$ is an $n$-OAF ring if and only if $R[X]$ is an OAF-ring if and only if $R[[X]]$ is an $n$-OAF-ring if and only if $R$ is a Noetherian von Neumann regular ring if and only if $R$ is a finite direct product of fields. Section 4 is dedicated to the study of the $n$-OAF property in the trivial extension $R=A\propto E$ of an $A$-module $E$. In particular, we show in Proposition \ref{pdiv1}, that a local ring $R$ with a unique nilpotent maximal ideal $M$ such that $M^n$ is divisible is an $n$-OAF ring. Also, we provide an example of an $n$-OAF ring that is not an OAF ring (see Example \ref{ex1}). Also, we investigate the conditions under which $R=A\propto E$ is an $n$-OAF ring for local and non-local cases (See, Theorem \ref{idealization}, Theorem \ref{generalid}, Corollary \ref{cid} and Corollary \ref{cid2}).

\section{Properties of $n$-OAF rings}
Throughout this paper, $R$ will always indicate a commutative ring with $1\neq 0$ and $\mathbb{N}$ will denote the set of positive integers. We begin with our first result, which will be frequently used in the sequel. 
\begin{lem} \label{lmp}
Let $R$ be a ring with Jacobson radical $M$, $I$ be an ideal of $R$ and $n\in\mathbb{N}$. The following statements are satisfied.
\begin{enumerate}
\item If $R$ is not a local ring, then $I$ is an $n$-OA ideal if and only if $I$  is a prime ideal.
\item If $R$ is a local ring, then $I$ is an $n$-OA ideal  if and only if $I$ is a prime ideal or $M^n \subset I \subset M$.
\end{enumerate}
\begin{proof}
\begin{enumerate}
\item Follows from \cite[Theorem 2.10]{BEM}.
\item Let  $I$ be an $n$-OA ideal of $R$. Suppose that $I$ is not prime, there exist $a,b \in M \backslash I$ such that $ab\in I$. Let $x_1,x_2,\ldots,x_n \in M$. Then we have $x_1\cdots x_{n-1}(x_na)b\in I$. Since $b \notin I$ and $I$ is an $n$-OA ideal, we obtain $x_1\cdots x_na \in I$. Since $a \notin I$ and $I$ is an $n$-OA ideal, we have $x_1\cdots x_n \in I$. As $I$ is not a prime ideal, we conclude that $M^n \subset I\subset M$. The converse is clear.
\end{enumerate}
\end{proof}

\end{lem}
\begin{prop}\label{min}
Let $n\in\mathbb{N}$ and $I$ be an $n$-OA ideal of a ring $R$. Then $|Min(I)|=1$.

\end{prop}
\begin{proof}
If $R$ is not a local ring, then by Lemma \ref{lmp} (1), $I$ is a prime ideal which implies that $|Min(I)|=1$. Now, we may assume that that $R$ is a local ring. Since $I$ is an $n$-OA ideal, by Lemma \ref{lmp}(2), $I$ is a primary ideal of $R$, and thus $\sqrt{I}=P$ is a prime ideal. Then, clearly we have $|Min(I)|=1$.
\end{proof}\label{promin}
\begin{prop}
Let $I$ be a proper ideal of a ring $R$ and $n\in\mathbb{N}$. If $I$ has an $n$-OA-factorization, then $|Min(I)|$  is finite.
\end{prop}
\begin{proof}
Let $I=\prod_{i=1}^mI_i$ be an $n$-OA-factorization. It is easy to see that $Min(I) \subseteq\cup_{i=1}^{m}Min(I_i)$, then by Proposition \ref{min}, we conclude that $|Min(I)|\leq n$.
\end{proof}
Let $R$ be a ring and $I$ be an ideal of $R$. Then, $I$ is called a multiplication ideal of $R$ if for each ideal $J$ of $R$ with $ J  \subseteq I$,  there exists an ideal $L$ of $R$ such that $J=IL$ \cite{LARS}.
\begin{lem}\label{lm2}
Let $n\in\mathbb{N}$ and $R$ be a local ring  whose each proper principal ideal has an $n$-OA-factorization. Then each nonmaximal minimal prime ideal of 
$R$ is principal. 

\end{lem}
\begin{proof}
Let $P$ be a nonmaximal minimal prime ideal of $R$. By \cite[Theorem 1]{and1}, it is suﬃcient to show that $P$ is a multiplication ideal of $R$. Let $x \in P$ and let $xR=\prod_{i=1}^mI_i$ be an $n$-OA-factorization. There is some $j \in [1,m]$ such that $I_j \subseteq P$. By Lemma \ref{lmp}(2), we have that $P=I_j$ and hence $xR=PJ$ for some ideal $J$ of $R$. We infer that $xR=P(xR:P)$. Now let $I$ be an ideal of $R$ such that $I \subseteq P$. Then $I=\sum_{y \in I}yR=\sum_{y \in I}P(yR:P)=P\sum_{y \in I}(yR:P)$ and thus $P$ is a multiplication ideal of $R$.
\end{proof}
\begin{lem}\label{lem3} Let $R$ be a local ring with maximal ideal $M$, $Nil(R)$ its nilradical and $n\in\mathbb{N}$. If every principal ideal of $R$ has an $n$-OA-factorization, then every principal ideal of $R/Nil(R)$ has an $n$-OA-factorization.
\end{lem}

\begin{proof}
By Lemma \ref{lmp} and Lemma \ref{lm2}, $Min(R)$ is finite and each $P \in Min(R)$ is a principal ideal of $R$. Let $I$ be a proper principal ideal of $R / Nil(R)$. Then $I=(x R+Nil(R)) / Nil(R)$ for some $x \in M$. Let $x R=\prod_{i=1}^m I_i$ be an $n$-OA-factorization. We infer that $I=(x R) / Nil(R)=\left(\prod_{i=1}^m I_i\right) / Nil(R)=\prod_{i=1}^m\left(I_i / Nil(R)\right)$. It suffices to show that $I_i / Nil(R)$ is an $n$-OA ideal of $R / Nil(R)$ for each $i \in[1, m]$. Let $i \in[1, m]$, if $I_i$ is a prime ideal of $R$, then $Nil(R) \subseteq I_i$, and hence $I_i / Nil(R)$ is a prime ideal of $R / Nil(R)$. Now let $I_i$ be a non-prime ideal of $R$. By Lemma \ref{lmp}(2), we have that $M^n \subseteq I_i \subseteq M$. Note that $R / Nil(R)$ is local with maximal ideal $M / Nil(R)$. Since $(M / Nil(R))^n=M^n / Nil(R) \subseteq I_i / Nil(R) \subseteq$ $M / Nil(R)$, then  by Lemma \ref{lmp}(2), $I_i / Nil(R)$ is an $n$-OA ideal of $R / Nil(R)$.
\end{proof}

Recall from \cite{GILM} that a ring $R$ is said to be a $\pi$-ring if its each principal ideal is a product of prime ideals.
\begin{prop}\label{prps}
Let $n\in\mathbb{N}$ and $R$ be a local ring with maximal ideal $M$ such that every proper principal ideal of $R$ has an $n$-OA-factorization with  $dim(R) = 1$. Then, $R$ is an integral domain.
\end{prop}
\begin{proof}
 If every $n$-OA ideal  of $R$ is a prime ideal, then $R$ is a $\pi$-ring, and hence $R$ is an integral domain by \cite[Theorem 46.8]{GILM}. Now, assume that there exists an $n$-OA ideal of $R$ which is not a prime ideal.  Set $L=M^n \cup \bigcup_{Q \in Min(R)} Q$. Let $x \in R \backslash L$ be a nonunit element of $R$. We show that $M^{n(n-1)} \subseteq xR$ for each $x \in R \backslash L$. Since $x \notin M^n $,  then $xR=\prod_{i=1}^{j}I_i$ with $1\leq j \leq n-1$ and $I_i$ is an $n$-OA ideal of $R$. Since $x\notin L$, $I_i$ can not be a minimal prime ideal for each $1\leq i\leq j$. If there exists $1\leq i \leq j$ such that $I_i$ is prime, then $I_i=M$ since $dim(R)=1$. If all $I_i$'s are $M$, then $xR=M^j\supseteq M^{n(n-1)}$. If there exists $1\leq i\leq j$ such that $I_i$ is not prime, then by Lemma \ref{lmp}(2),  we infer that $M^n \subseteq I_i\subseteq M$ which implies that $M^{n(n-1)}\subseteq M^{nj} \subseteq xR$.
Now, we prove that $P \subseteq M^{n}$ for each $P \in Min(R)$. Let $P \in Min(R)$ with $P \nsubseteq M^{n}$. Choose $w \in R \backslash P$. Then $P+wR \nsubseteq L $ by the prime avoidance lemma. Let $v \in(P+w R) \backslash L$,  thus  $M^{n(n-1)} \subseteq v R \subseteq P+w R$. Since $P$ is a nonmaximal prime ideal, we have that $R / P$ has no simple $R / P$-submodules, and hence $\bigcap_{y \in R \backslash P}(P+y R)=P$. (Note that if $\bigcap_{y \in R \backslash P}(P+y R) \neq P$, then $\bigcap_{y \in R \backslash P}(P+y R) / P$ is a simple $R / P$-submodule of $R / P$). This implies that $M^{n(n-1)} \subseteq \bigcap_{y \in R \backslash P}(P+y R)=P$, and thus $P=M$, a contradiction.\\

Let $Q\in Min(R)$. By above argument, we have $Q\subseteq M^n$. Since $R$ has an $n$-OA ideal which is not a prime ideal, by Lemma \ref{lmp} (2), we have $M^n\subset M$ which implies that $M^2\subset M$. Choose $z\in M-M^2$. By the assumption, $Rz=H_1H_2\cdots H_t$ for some $n$-OA ideals $H_i$ of $R$. As $z\notin M^2$, it follows that $t=1$, that is $Rz=H_1$ is an $n$-OA ideal of $R$. If $Rz$ is prime, then clearly $Rz=M$ as $dim(R)=1$. If $Rz$ is not a prime ideal, then by Lemma \ref{lmp} (2), we have $M^n\subset Rz\subset M$. By above cases, we conclude that $Q\subseteq M^n\subset Rz$. This gives $Q=zQ$. On the other hand, by Lemma \ref{lm2}, $Q$ is principal. Let $Q=Ry$ for some $y\in Q$. By the equality $Q=Ry=Qz=Ryz$, we can write $y=ayz$ for some $a\in R$. Since $1-az$ is a unit and $y(1-az)=0$, we obtain $y=0$ which implies that $Q=Ry=(0)$. Consequently, $R$ is a domain. 

\end{proof}
Now we need the following lemma to generalize the previous result.
\begin{lem}\label{lem4}
Let $R$ be a local ring with maximal ideal $M$ such that every proper principal ideal of $R$ has an $n$-OA-factorization and $dim(R)\geq 2$ and $R$ is reduced, where $n\in\mathbb{N}$. Then, $R$ is a unique factorization domain.
\end{lem}
\begin{proof}
Since $dim(R)\geq 2$, there exists a nonmaximal nonminimal prime ideal $Q$ of $R$. Also, by the prime avoidance lemma, there exists $x \in Q \backslash  \bigcup_{P \in Min(R)}P$. We prove that $x$ is a regular element, since $R$ is reduced, we have that $\bigcap_{P \in Min(R)} P=0$. Suppose that $x$ is a zero divisor. Then there exists $0\neq y\in R$ such that $xy=0$. As $y\neq 0$ and $R$ is reduced, there exists a minimal prime ideal $P$ of $R$ with $y\notin P$. Since $xy=0\in P$, we conclude that $x\in P$ which is a contradiction.\\
Let $xR=\prod_{i=1}^m I_i$ be an $n$-OA-factorization, then $I_j \subseteq Q$ for some $j \in [1,m]$. Since $x$ is regular, $I_j$ is invertible and hence $I_j$ is a regular principal ideal (because invertible ideal of a local ring are regular principal ideals). Since $I_j \subseteq Q$ and $Q \neq M$, we have that $I_j$ is a prime ideal by Lemma \ref{lmp}(2). Consequently, $P \subseteq I_j$ for some $P \in Min(R)$, since $I_j$ is regular, we infer that $P \subset I_j$. As $I_j$ is a principal ideal, we have $P=PI_j$, then by Nakayama's lemma, we obtain that $P=(0)$, that is $R$ is an integral domain.\\
Now, we show that $R$ is a unique factorization  domain, by \cite[Theorem 2.6]{HS}, it suﬃces to show   that every nonzero prime ideal of $R$ contains a nonzero principal prime ideal. Since $dim(R)\geq 2$  and $R$ is local, we only need to show that every  nonzero nonmaximal prime ideal of $R$ contains a  nonzero principal prime ideal. Let $Q$ be a nonzero nonmaximal prime ideal of $R$ and let $z \in  Q$ be  nonzero. Then we can write $zR=\prod_{k=1}^{s}J_k$ for some $n$-OA ideals $J_k$ of $R$. Then $J_k \subseteq Q$ for some $k \in [1,s]$. Since $R$ is an integral domain, $J_k$ is invertible and hence $J_k$ is nonzero and principal since $R$ is local. As $Q \neq M$,  then $J_k$ is a prime ideal by Lemma \ref{lmp}(2). Then, $R$ is a unique factorization domain.
\end{proof}
\begin{prop}\label{prp1}
Let $R$ be a local ring with maximal ideal $M$ such that $dim(R) \geq 2$ and every proper principal ideal of $R$ has an $n$-OA-factorization, where $n\in\mathbb{N}$. Then $R$ is a unique factorization domain.
\end{prop}
\begin{proof}
We will show that $R$ is a unique factorization domain. Note that $R/Nil(R)$ is a reduced local ring with maximal ideal $M/Nil(R)$ and\\ $dim(R/Nil(R)) \geq 2$. Moreover, each proper principal ideal of $R/Nil(R)$ has an $n$-OA-factorization by Lemma \ref{lem3}. It follows by Lemma \ref{lem4} that $R/Nil(R)$ is a unique factorization domain, and thus $Nil(R)$ is the unique minimal prime ideal of $R$. Since $R/Nil(R)$ is a unique factorization domain and $dim(R/Nil(R))\geq 2$, $R/Nil(R)$ possesses a nonzero nonmaximal principal prime ideal. We infer that there is  some nonminimal nonmaximal prime ideal $Q$ of $R$ such that $Q/Nil(R)$ is a principal ideal of $R/Nil(R)$. Consequently, there exists $q\in Q$ such that $Q=qR+Nil(R)$. Let $qR=\prod_{i=1}^{r}I_i$ be an $n$-OA-factorization. Then $I_j \subseteq Q$ for some $j\in [1,r]$. Since $Q\neq M$, we infer by Lemma \ref{lmp}(2) that $I_j$ is prime ideal of $R$. Therefore, $Q=qR+Nil(R)\subseteq I_j \subseteq Q$, and hence $I_j=Q$.\\
Assume that $Q\neq qR$. Then $q R=Q J$ for some proper ideal $J$ of $R$. It follows that $q \in q R=(q R+Nil(R)) J \subseteq q J+Nil(R)$, and thus $q(1-a) \in Nil(R)$ for some $a \in J$. Since $a$ is a nonunit of $R$, we obtain that $q \in Nil(R)$. This implies that $Q=qR+Nil(R)=Nil(R)$ which is a contradiction. We infer that $Q=q R$. Since $Nil(R)\subset Q$ and $Nil(R)$ is a prime ideal of $R$, we have that $Nil(R)=Nil(R)Q$. On the other hand, by Lemma \ref{lm2}, $Nil(R)$ is a principal ideal. Since $Nil(R)=Nil(R)Q$, by Nakayama's lemma, we have $Nil(R)=0$, and thus $R \cong R / Nil(R)$ is a unique factorization domain.                   

\end{proof}
\begin{prop}\label{prp3}
Let $R$ be a local ring with maximal ideal $M$ such that each proper $2$-generated ideal of $R$ has an $n$-OA-factorization, where $n\in\mathbb{N}$. Then $dim(R) \leq 2$ and each nonmaximal prime ideal of $R$ is principal.
\end{prop}
\begin{proof}
To show that $dim(R) \leq 2$, it is sufficient to prove that $dim(R_P) \leq 1$ for each nonmaximal prime ideal $P$ of $R$. Let $P$ be a nonmaximal prime ideal and let $I$ be a $2$-generated ideal of $R_P$. It is easy to see that $I=J_P$ for some $2$-generated ideal $J$ of $R$ with $J \subseteq P$. Let $J=\prod_{i=1}^{m}J_i$ be an $n$-OA-factorization, thus $I=J_P=\prod_{i=1, J_i  \subseteq P}^{m}(J_i)_P$.\\ Let $i \in [1 , n]$ such that $J_i \subseteq P$. Then $J_i$ is a prime ideal of $R$ by Lemma \ref{lmp}(2), and thus $(J_i)_P$ is a prime ideal of $R_P$. Thus we conclude that $I$ is a product of prime ideal of $R_P$. Therefore $R_P$ is a general ZPI-ring by \cite[Theorem 3.2]{LEV}. Also by \cite[Page 205]{LARS}, we have $dim(R_P) \leq 1$. This implies that $dim(R) \leq 2$.\\
Now we shall prove that, each nonmaximal prime ideal of $R$ is principal. If $dim(R)=0$, there is nothing to prove since $R$ has no nonmaximal prime. Assume that $dim(R) \geq 1$. Then by Proposition \ref{prp1} and Proposition \ref{prps}, $R$ is either a one dimensional domain or a two dimensional unique factorization domain. Then each nonmaximal prime ideal of $R$ is principal.  
\end{proof}
\begin{thm}\label{thm12}
Let $n\geq 2$ be a positive integer, $R$ be a ring such that every proper $2$-generated ideal of $R$ has an $n$-OA-factorization. Then $dim(R) \leq 1$. In particular, if $dim(R)=1$, then $R$ is a domain.
\end{thm}
\begin{proof}
 First observe that, if every $n$-OA ideal of $R$ is a prime ideal, then $R$ is a general ZPI-ring by \cite[Theorem 3.2]{LEV}, and hence $dim(R) \leq 1$ by \cite[Page 205]{LARS}. Now, assume that there exists an $n$-OA ideal which is not a prime ideal. Then by Lemma \ref{lmp}, $R$ is local ring with unique maximal ideal $M$ such that $M^n\subset M$ for every $n\geq 2$. It suffices to show that if $Q$ is a nonmaximal prime ideal of $R$, then $Q=0$.\\ Let $Q$ be a nonmaximal prime ideal of $R$. We will show that $Q\subset M^n$. Assume that $Q \nsubseteq M^{n}$. Since $dim(R) \leq 2$ by Proposition \ref{prp3}, there is a prime ideal $P$ of $R$ such that $Q \subseteq P$ and $dim(R / P)=1$. Next we prove that $M^{n} \subseteq P+y R$ for each $y \in R \backslash P$. Let $y \in R \backslash P$ and put $J=P+y^2 R$. Without loss of generality, we may assume that $J \subset M$. Choose minimal prime ideal $P^{*}$ of $J$. If $J=P^{*}$, then $P^{*}=P+Ry$ since $y\in P^{*}$. Then we have $P+Ry=P+Ry^2$ which implies that $y(1-ay)\in P$. Since $1-ay$ is unit, we conclude that $y\in P$ which is a contradiction. Thus we conclude that a chain $Q\subseteq P\subset J\subset P^{*}\subseteq M$. As $dim(R)\leq 2$, we have $M$ is the minimal prime ideal of $J$, that is, $P^{*}=M$. On the other hand, by Proposition \ref{prp3}, it follows that $J$ is 2-generated. Since $J \nsubseteq M^{n}$, then $J=\prod_{i=1}^{j}I_i$  with $1 \leq j \leq n-1$ and $I_i$ is a $n$-OA ideal of $R$. Note that $J \subset I_i \subseteq M$ for every $1\leq j\leq n-1$. If $I_i$ is prime, then note that $I_i=M$ since $M$ is the minimal prime of $J$. If all $I_i$'s are prime, then clearly $J=M^{n-1}\supseteq M^n$. Assume $I_i$ is not prime for some $1\leq i\leq n-1$. Then by Lemma \ref{lmp} (2), we have $M^n\subset I_i\subset M$ which implies that $M^{n(n-1)}\subseteq M^{nj}\subseteq J\subseteq P+Ry$. Moreover, $R/P$ is an integral domain that is not a field. Consequently, $R / P$ does not have any simple $R / P$-submodules, which implies that $P=\bigcap_{x \in R \backslash P}(P+x R)$. (Observe that if $\bigcap_{x \in R \backslash P}(P+x R) \neq P$, then $\bigcap_{x \in R \backslash P}(P+x R) / P$ is a simple $R / P$-submodule of $R / P$). Therefore, $M^{n(n-1)} \subseteq \bigcap_{x \in R \backslash P}(P+x R)=P$, and hence $P=M$, a contradiction. We infer that $Q \subseteq M^{n}$.

Since $M^2\subset M$, there exists $z\in M-M^2$. Then $ Rz=\prod_{i=1}^{j}I_i$ has an $n$-OA factorization. Since $z\notin M^2$, $j=1$, that is $Rz$ is an $n$-OA ideal of $R$. Then by Lemma \ref{lmp} (2), we conclude that $M^n\subset Rz$. Which implies that $Q\subseteq M^n\subset Rz$, and thus we have $Q=zQ$. On the other hand, $Q$ is principal by Proposition \ref{prp3}. Then by Nakayama's lemma, we have $Q=0$. The rest is clear.
\end{proof}
\begin{prop}\label{prp5}
Let $n\in\mathbb{N}$ and $I$ be a proper ideal of an $n$-OAF-ring. Then $R/I$ is an $n$-OAF-ring.
\end{prop}
\begin{proof}
Let $J$ be a proper ideal of $R$ which contains $I$. Let $J=\prod_{i=1}^m J_i$ be an $n$-OA-factorization. Then $J / I=\prod_{i=1}^m\left(J_i / I\right)$. It suffices to show that $J_i / I$ is an $n$-OA ideal for each $i \in[1, m]$. Let $i \in[1, m]$ and let $a_1, a_2, \ldots, a_{n+1} \in R$ be such that $\overline{a_1}, \overline{a_2}, \ldots, \overline{a_{n+1}}$ are nonunits of $R / I$ and $\overline{a_1} \overline{a_2}\cdots \overline{a_{n+1}} \in J_i / I$. Clearly, $a_1, a_2, \ldots,a_{n+1}$ are nonunits of $R$ and $a_1a_2\cdots a_{n+1} \in J_i$. Since $J_i$ is an $n$-OA ideal of $R$, we conclude that $a_1a_2\cdots a_{n} \in J_i$ or $a_{n+1} \in J_i$ which implies that $\overline{a_1} \overline{a_2}\cdots \overline{a_n} \in J_i / I$ or $\overline{a_{n+1}} \in J_i / I$. Therefore, $R / I$ is an $n$-OAF-ring.
\end{proof}

Recall from \cite{Kim} that a nonempty subset $S$ of a ring $R$ is said to be a multiplicatively closed set if $0\notin S$, $1\in S$ and $ab\in S$ for every $a,b\in S$.
\begin{prop}\label{ploc}
Let $S$ be a multiplicatively closed subset of $R$. If $R$ is an $n$-OAF-ring, then $S^{-1} R$ is an $n$-OAF-ring. In particular, $R_M$ is an $n$-OAF-ring for every maximal ideal $M$ of $R$.
\end{prop}
\begin{proof}
 Let $J$ be a proper ideal of $S^{-1} R$. Then $J=S^{-1} I$ for some proper ideal $I$ of $R$ with $I \cap S=\varnothing$. Let $I=\prod_{i=1}^m I_i$ be an $n$-OA-factorization. Assume that $I_i\cap S=\emptyset$ for every $i\in[1,t]$ and $I_i\cap S\neq\emptyset$ for every $i\in[t+1,m]$. Then $J=\prod_{i=1}^t\left(S^{-1} I_i\right)$ where each $S^{-1} I_i$ is an $n$-OA ideal by \cite[Theorem 11]{gulak}. Then $S^{-1} R$ is an $n$-OAF-ring. The in particular statement is clear.
\end{proof}

\begin{prop}\label{car}
Let $R=A\times B$, where $A,B$ are two rings, and $n\in\mathbb{N}$. Then $R$ is an $n$-OAF-ring if and only if $A,B$ are general ZPI-rings.
\end{prop}

\begin{proof}
Assume that $R$ is an $n$-OAF-ring. Let $I$ be a proper ideal of $A$. Since $R$ is an $n$-OAF-ring, $I\times B=\prod_{i=1}^m J_i$ has an $n$-OA factorization of $I\times B$. As $R$ is not a local ring, by Lemma \ref{lmp} (2), $J_i$'s are all prime ideals of $R$, and thus $J_i=P_i\times B$ for some prime ideal $P_i$ of $A$. This gives $I\times B=(P_1P_2\cdots P_m)\times B$ which implies that $I=P_1P_2\cdots P_m$, that is $A$ is a general ZPI-ring. Likewise, $B$ is a general ZPI-ring. For the converse, assume that $A,B$ are general ZPI-rings. Let $I\times J$ be a proper ideal of $R$, where $I$ is an ideal of $A$ and $J$ is an ideal of $B$. Then without loss of generality, we may assume that $I,J$ are proper. As $A,B$ are general ZPI-rings, we can write $I=P_1P_2\cdots P_m$ and $J=Q_1Q_2\cdots Q_t$ for some prime ideals $P_i$ of $A$ and $Q_j$ of $B$. This implies that $I\times J=(P_1\times B)(P_2\times B)\cdots (P_m\times B)(A\times Q_1)(A\times Q_2)\cdots (A\times Q_t)$ a finite product of primes, and hence an $n$-OA factorization of $I\times J$. Consequently, $R$ is an $n$-OAF-ring.
 \end{proof}

\begin{cor}\label{ccar}
Let $R=A_1\times A_2\times\cdots\times A_m$, where $m\geq 1$, $A_i$'s are rings. Let $n\in\mathbb{N}$. Then $R$ is an $n$-OAF-ring if and only if $A_i$ is a general ZPI-ring for every $i=1,2,\ldots,m$.
\end{cor}
\begin{proof}
It can be easily verified by using Proposition \ref{car} and mathematical induction.
\end{proof}

\section{Polynomial ring, Formal power series ring and the ring $A+XB[X]$}
The conditions under which the polynomial ring $R[X]$, the formal power series ring $R[[X]]$, and the ring $A+XB[X]$ are $n$-OAF-rings are now determined. The following result is required first.

\begin{lem}\label{rm}
Let $R$ be a ring and $n\in\mathbb{N}$. The following statements are satisfied.
\begin{enumerate}
\item If $R$ is not a local ring, then $R$ is a general ZPI ring if and only if $R$ is an $n$-OAF-ring.
\item If $R[[X]]$ or $R[X]$ is an $n$-OAF-ring, then $R$ is an $n$-OAF-ring.
\end{enumerate}
\end{lem}
\begin{proof}
\begin{enumerate}
    \item Follows from Lemma \ref{lmp} (1).
    \item If $R[[X]]$ or $R[X]$ is an $n$-OAF-ring, then by Proposition \ref{prp5} and the isomorphisms $R[[X]]/(X)\cong R\cong R[X]/(X)$, we have the desired result.
\end{enumerate}
\end{proof}

Recall from \cite{von} that a commutative ring $R$ is said to be a \textit{von Neumann regular ring} if for each $a\in R$ there exists $x\in R$ such that $a=a^2x$. It is well known that $R$ is a von Neumann regular ring if and only if its every proper ideal is idempotent, that is $I=I^2$ if and only if its every proper ideal is a radical ideal, that is $I=\sqrt{I}$ \cite{JaTe}.
\begin{prop}\label{poly}
Let $R$ be a ring and $n\in\mathbb{N}$. The following are equivalent.
\begin{enumerate}
    \item $R[X]$ is an $n$-OAF-ring.
    \item $R[X]$ is an OAF-ring. 
    \item $R$ is a Noetherian von Neumann regular ring. 
    \item $R$ is a finite direct product of fields. 
\end{enumerate}
\end{prop}

\begin{proof}
$(1)\Rightarrow(2)$ Assume that $R[X]$ is an $n$-OAF ring. Since $R[X]$ is not a local ring, by Lemma \ref{rm}, $R[X]$ is a general ZPI-ring. Thus by \cite[Theorem 2.8]{ABS1}, $R[X]$ is an OAF-ring.\\
$(2)\Rightarrow (3)\Rightarrow (4)$ Follows from \cite[Corollary 2.6]{ABS1}.\\
$(4)\Rightarrow (1)$ Let $R$ be a finite direct product of fields, that is $R=F_1\times F_2\times\cdots\times F_n$ where $F_i$'s are fields. Then $R[X]\cong F_1[X]\times F_2[X]\times\cdots\times F_n[X]$. Since $F_i$ is a field, $F_i[X]$ is a PID which is clearly a general ZPI-ring. Then by Corollary \ref{ccar}, $R[X]$ is an $n$-OAF ring.
\end{proof}
Recall from \cite{WOD} that an ideal $I$ of $R$ is said to be a simple ideal if $I^2\subseteq J\subseteq I$ for some ideal $J$ of $R$, then $J=I^2$ or $J=I$. 

\begin{lem}\label{impt}
Let $R$ be a  ring. Then  $M+XR[[X]]$ is a simple ideal of $R[[X]]$ if and only if $M^2=M$ for each $M \in Max(R)$.
\end{lem}
\begin{proof}
$(\Rightarrow)$ Let $M \in Max(R)$ such that $M+XR[[X]]$ is a simple ideal of $R[[X]]$. Then note that $(M+X R[[X]])^2=M^2+XM[[X]]+X^2 R[[X]]$. Put $I=M+XM[[X]]+X^2R[[X]]$. Then we have $I$ is an ideal of $R[[X]]$ and $(M+X R[[X]])^2 \subseteq I \subsetneq M+XR[[X]]$ which implies that $M=M^2$.\\
$(\Leftarrow)$ Let $I$ be an ideal of $R[[X]]$ such that $(M+X R[[X]])^2 \subseteq I \subseteq M+X R[[X]]$ and let $I_1=\left\{a \in R \mid\right.$ there exists $f=\sum_{i=0}^{\infty} a_i X^i \in I$ with $\left.a_1=a\right\}$. Then $I_1$ is an ideal of $R$ and $M \subseteq I_1 \subseteq R$. Assume that $I_1=M$. Now we show that $I=M+X M[[X]]+X^2 R[[X]]$. In fact $M+X M[[X]]+X^2 R[[X]] \subseteq I$. Conversely, let $f=\sum_{i=0}^{\infty} a_i X^i \in I \subseteq M+X R[[X]]$, then $a_1 \in I_1=M$ so $f \in M+X M[[X]]+X^2 R[[X]]$. This shows that $I=(M+XR[[X]])^2$. If $I_1=R$, then we show that $I=M+X R[[X]]$. In fact, $I \subseteq M+X R[[X]]$. Conversely, $M+X M[[X]]+X^2 R[[X]] \subseteq I$, so $M \subseteq I$ and $X^2 R[[X]] \subseteq I$. We prove that $X \in I$. As $I_1=R$, then $1 \in I_1$, so there exists $g=\sum_{i=0}^{\infty} b_i X^i \in I$ with $b_1=1$, we conclude that  $X=g-b_0-\sum_{i=2}^{\infty} b_i X^i\in I$. Then $XR[[X]]\subseteq I$ and $M+XR[[X]] \subseteq I$ which implies that $I=M+XR[[X]]$.
\end{proof}

\begin{thm}\label{series}
Let $R$ be a ring and $n\in\mathbb{N}$. Then $R[[X]]$ is an $n$-OAF-ring if and only if $R$ is a finite direct product of fields.
\end{thm}
\begin{proof}
$(\Rightarrow)$ Assume that $R[[X]]$ is an $n$-OAF-ring. First note that\\ $Max(R[[X]])=\{M+XR[[X]]\ |\ M \in Max(R) \}$. Assume that $R$ is a non local ring, then $R[[X]]$ is a non local ring. By Lemma \ref{rm}, $R[[X]]$ is a general ZPI-ring. Then by \cite[Theorem 3]{WOD}, $R[[X]]$ is Noetherian and for each $M \in Max(R)$, $M+XR[[X]]$ are simple ideals. Which implies that $R$ is Noetherian and  $M+XR[[X]]$ are simple for each $M \in Max(R)$. By Lemma \ref{impt} $M^2=M$ for each $M \in Max(R)$, which implies that $R$ is isomorphic to a finite product of fields by \cite[Theorem 3.2]{THH}. Now, assume that $R$ is a local ring. Now we will show that $R$ is a field. As $dim(R[[X]]) \geq dim(R)+1$ and $R[[X]]$ is an $n$-OAF ring, by Theorem \ref{thm12}, $R[[X]]$ is an integral domain (so $R$ is an integral domain) with $dim(R[[X]])=1$ (so $dim(R)=0$). Therefore $R$ is an integral domain with $dim(R)=0$, that is, $R$ is a field.\\
$(\Leftarrow)$ Let $R$ be a finite direct product of fields, that is, $R=F_1\times F_2\times\cdots\times F_m$, where $m\geq 1$ and $F_i$'s are all fields. Then by Chinese remainder theorem, $R[[X]]\cong F_1[[X]]\times F_2[[X]]\times\cdots\times F_m[[X]]$. As $F_i$ is a field, $F_i[[X]]$ is a PID which is clearly an $n$-OAF ring. Then by Corollary \ref{ccar}, $R[[X]]\cong F_1[[X]]\times F_2[[X]]\times\cdots\times F_m[[X]]$ is an $n$-OAF-ring.

\end{proof}
Now, we study the transfer of the $n$-OAF property of the ring $A+XB[X]$.
\begin{prop}\label{extension}
Let $n\in\mathbb{N}$ and $A \subseteq B$ be an extension of commutative rings. If $A+X B[X]$ is an $n$-OAF-ring, then $A$ is a finite direct product of fields.
\end{prop}
\begin{proof}
 If $A+X B[X]$ is an $n$-OAF ring, then $A$ is Noetherian. We prove first that $dim(A)=0$. Note that if $A+X B[X]$ is an $n$-OAF-ring, then it is a general ZPI-ring, so by \cite[Page 225, Exercise 10]{LARS}, it is also a multiplication ring, so $dim(A+XB[X])\leq 1$, by \cite[Page 210]{LARS}. If $p_1 \subset p_2 \subset \cdots \subset p_n$ is a chain of prime ideals of $A$, then $p_1+X B[X] \subset p_2+X B[X] \subset \cdots \subset p_n+X B[X]$ is a chain of prime ideals of $A+X B[X]$, which implies that $\operatorname{dim} A \leq \operatorname{dim}(A+X B[X])$, then $\operatorname{dim} A \leq 1$. Let $m \in \operatorname{Max}(A)$ and suppose that there exists $p \in \operatorname{Spec}(A)$ such that $p \subsetneq m$, then $M=m+X B[X] \in$ $\operatorname{Max}(A+X B[X]), P=p+X B[X] \in \operatorname{Spec}(A+X B[X])$ and $P \subsetneq M$. By \cite[Proposition 9.15]{LARS}, $P=\bigcap_{n=1}^{\infty} M^n \subset M^2=m^2+mXB[X]+X^2B[X]$, so $B \subseteq m B$. But $A \subseteq B$ is an integral extension of rings as $B$ is a finitely generated $A$- module. So there exists $q \in \operatorname{Spec}(B)$ such that $m \subseteq q$, then $m B \subseteq q \subsetneq B$ which is a contradiction. So $dim(A)=0$. As $A$ is Noetherian, we deduce that $A$ is Artinian. Moreover, for each $m \in \operatorname{Max}(A)$ the maximal ideal $M=m+X B[X]$ of $A+X B[X]$ is simple. Take $I=m+mXB[X]+X^2 B[X]$, then $I$ is an ideal of $A+X B[X]$ and $M^2 \subseteq I \subsetneq M$, so $M^2=I$ and then $m^2=m$. By \cite[Theorem 3.2]{THH}, $A$ is isomorphic to a finite direct product of fields.

\end{proof}

\section{$n$-OA factorization in trivial extension}
This section is dedicated to the study of $n$-OA factorization in the trivial extension $A\propto E$ of an $A$-module $E$. Let $E$ be an $A$-module. The trivial extension $A\propto E=A\bigoplus E$ is a commutative ring with componentwise addition and the multiplication $(a,e)(b,f)=(ab,af+be)$ for every $(a,e),(b,f)\in A\propto E$, \cite{idealization} and \cite{Huckaba}. If $I$ is an ideal of $A$ and $V$ is a submodule of $E$, then $I\propto V$ is an ideal of $A\propto E$ if and only if $IE\subseteq V$ \cite[Theorem 3.1]{idealization}. In this case, $I\propto V$ is called a homogeneous ideal of $A\propto E$. Also, every prime (maximal) ideal of $A\propto E$ has the form $I\propto E$ for some prime (maximal) ideal $I$ of $A$ \cite[Theorem 3.2]{idealization}. Also, $dim(A\propto E)=dim(A)$.\\

\begin{prop}\label{pdiv1} Let $R$ be a local ring with unique maximal ideal $M$ and $n\in\mathbb{N}$. Suppose that $M^n$ is a divided ideal and $M$ is nilpotent. Then $R$ is an $n$-OAF ring.
\end{prop}

\begin{proof}
Since $M$ is nilpotent, there exists $k\in\mathbb{N}$ such that $M^k=(0)$. If $k\leq n$, then $M^n=(0)$ which implies that every ideal of $R$ is an $n$-OA ideal, and hence $R$ is trivially an $n$-OAF-ring. Thus, we assume that $k>n$. On the other hand, if $M^n=M^{n-1}$, then clearly $M^{n-1}=M^s$ for every $s\geq n-1$ which implies that $M^{n-1}=M^n=M^k=(0)$. Again by above argument, $R$ is an $n$-OAF ring. Suppose that $M^n\subset M^{n-1}$. This implies that $M^n\subset\ M^{n-1}\subset\cdots\subset M^3\subset M^2\subset M$. Let $J$ be an ideal of $R$. If $J=(0)$, then $J=(0)=M^k$ is an $n$-OA factorization. Suppose that $J$ is not a zero ideal of $R$. If $M^n\subseteq J$, then by Lemma \ref{lmp} (2), $J$ is an $n$-OA ideal of $R$. So suppose that $M^n\nsubseteq J$. Since $M^n$ is divided, we have $J\subseteq M^n$. Choose $x\in M-M^2$ (i.e. $x\notin M^n$). Since $M^n$ is divided, we conclude that $J\subseteq M^n\subset Rx$. Then there exists a proper ideal $I_1$ of $R$ such that $J=(Rx)I_1$. As $M^n\subseteq Rx$, $Rx$ is an $n$-OA ideal of $R$. If $M^n\subseteq I_1$, then $I_1$ is an $n$-OA ideal of $R$, and hence $J=(Rx)I_1$ is an $n$-OA factorization. Thus, we may assume that $M^n\nsubseteq I_1$ which implies that $I_1\subseteq M^n$. An argument analogous to the one above shows that $I_1=(Rx)I_2$ for some proper ideal $I_2$ of $R$. Then we have $J=(Rx)^2I_2$. If $M^n\subseteq I_2$, then $J=(Rx)^2I_2$ is an $n$-OA-factorization. So assume that $I_2\subseteq M^n$. If we proceed in this manner, we conclude that $J=(Rx)^{k-n}I_{k-n}$ for some proper ideal $I_{k-n}$ of $R$. If $I_{k-n}\subseteq M^n$, then we conclude that $J=(Rx)^{k-n}I_{k-n}\subseteq M^k=(0)$ which is a contradiction. Thus, we must have $I_{k-n}\nsubseteq M^n$, that is, $M^n\subseteq I_{k-n}$. This implies that $I_{k-n}$ is an $n$-OA ideal of $R$ and $J=(Rx)^{k-n}I_{k-n}$ is an $n$-OA factorization. Consequently, $R$ is an $n$-OAF-ring.
\end{proof}

\begin{exam}\label{ex1}
Let $R=\mathbb{Z}_2[X,Y]/I$ where $I=(X^2,XY,Y^4)$ and $X,Y$ are indeterminates. Then $R$ is a zero dimensional local ring with unique prime ideal $M=(x,y)$, where $x=X+I$ and $y=Y+I$. On the other hand, note that $M$ is nilpotent with $M^4=(0)$, $M^3=(y^3)$ and $M^2=(y^2)$. Since $(x)$ and $M^2$ are not comparable, $M^2$ is not divided. Then by \cite[Theorem 4.2]{ABS1}, $R$ can not be an OAF-ring. On the other hand, easy computer check shows that any ideal of $R$ containing $M^3$ are $\{R,(y),(y^2),(y^3),(x,y),(x,y^2),(x,y^3),(x+y),(x+y^2)\}$. By Lemma \ref{lmp} (2) shows that $\{(y),(y^2),(y^3),(x,y),(x,y^2),(x,y^3),(x+y),(x+y^2)\}$ are the only 3-OA ideals of $R$. It is easy to see that $(x)$ can not be written as a product of 3-OA ideals, and thus $R$ is not an $3$-OAF-ring. Also, since $M^4=(0)$, $R$ is a 4-OAF-ring. 
\end{exam}

Now we need the following lemma, which will be used in the main theorem of this section. 
\begin{lem}\cite{idealization}\label{lemid}
Let $E$ be an $A$-module and $R=A\propto E$. The following statements are satisfied. 
\begin{enumerate}
    \item $Nil(R)=Nil(A)\propto E$. Thus $R$ is an integral domain if and only if $A$ is an integral domain and $E=0$.
    \item Let $I,J$ be two ideals of $A$. Then $(I\propto E)(J\propto E)=IJ\propto (IE+JE)$.
    \item If $\mathcal{J}$ is an ideal of $R$ contains $(0)\propto E$, then $\mathcal{J}=I\propto E$ for some ideal $I$ of $A$.
    \item $R$ is a local ring if and only if $A$ is a local ring. 
    \item  $R$ is a general ZPI-ring if and only if $A$ is a general ZPI-ring and $E$ is cylic with $ann(E)=P_1P_2\cdots P_s$ for some idempotent maximal ideals $P_1,P_2,\ldots,P_s$ or $E=0$ (if $s=0$, then $ann(E)=R$).
\end{enumerate}
\end{lem}

\begin{proof}
Combine \cite[Theorem 3.1]{idealization}, \cite[Theorem 3.2]{idealization}, \cite[Theorem 3.3]{idealization} and \cite[Theorem 4.10]{idealization}
\end{proof}

\begin{thm}\label{idealization}
Let $E$ be an $A$-module, $R=A\propto E$ and $n\in\mathbb{N}$. Assume that $A$ is a non-local ring. Then $R$ is an $n$-OAF-ring if and only if $A$ is a general ZPI-ring and $E$ is cyclic with $ann(E)=P_1P_2\cdots P_s$ for some idempotent maximal ideals $P_1,P_2,\ldots,P_m$ or $E=0$ (if $s=0$, then $ann(E)=R$).
\end{thm}
\begin{proof}
Let $A$ be a non-local ring. Then by Lemma \ref{lemid} (4), $R$ is a non-local ring. Also, by Lemma \ref{rm}, $R$ is an $n$-OAF-ring if and only if $R$ is a general ZPI ring. The rest follows from Lemma \ref{lemid} (5).
\end{proof}

\begin{thm}\label{generalid}
Let $E$ be a module over a local ring $A$ with unique maximal ideal $M$, $R=A\propto E$ and $n\in\mathbb{N}$. The following statements are satisfied.
\begin{enumerate}
    \item $R$ is a 1-dimensional $n$-OAF ring if and only if $A$ is a 1-dimensional local $n$-OAF-domain and $E=0$.
    \item  If $R$ is a zero dimensional $n$-OAF ring, then $A$ is a zero dimensional $n$-OAF ring and either $E=0$ and $M$ is a nilpotent ideal or $M^n=(0)$. In the second case, every ideal of $A$ is an $n$-OA ideal. The converse also holds if $M^{n-1}E\subseteq Mx$ for every $x\notin M^{n-1}E$ and $M^{n-1}E\subseteq aE$ for every $0\neq a\in M$.
\end{enumerate}
\end{thm}

\begin{proof}
$(1)$ Let $A$ be a local ring with unique maximal ideal $M$. Then by Lemma \ref{lemid} (4), $R$ is a local ring with unique maximal ideal $M\propto E$. Assume that $R$ is a 1-dimensional $n$-OAF ring. Then by Theorem \ref{thm12}, $R$ is a domain. In this case, by Lemma \ref{lemid} (1), $A$ is a 1-dimensional domain and $E=0$. The converse is easy since $A\cong A\propto (0)=R$.\\
$(2)$ Assume that $R$ is a zero dimensional $n$-OAF ring. In this case, $M$ and $M\propto E$ are the unique prime ideals of $A$ and $R=A\propto E$, respectively. Since $R$ is an $n$-OAF ring, we can write $(0)\propto E=H_1H_2\cdots H_t$ for some $n$-OAF-ideals $H_1,H_2,\ldots,H_t$ of $R$. Since $(0)\propto E\subseteq H_i$, by Lemma \ref{lemid} (3), $H_i=J_i\propto E$ for some proper ideal $J_i$ of $A$. Then by Lemma \ref{lmp} (2), we conclude that either $J_i=M$ or $M^n\subseteq J_i\subseteq M$. On the other hand, by Lemma \ref{lemid} (2), we conclude that $J_1J_2\cdots J_t=(0)$ and $(J_1J_2\cdots J_{t-1}+J_1J_2\cdots J_{t-2}J_t+\cdots+J_2J_3\cdots J_t)E=E$. Assume that $t>1$. The second equality gives $E\subseteq M^{t-1}E\subseteq E$, that is, $E=ME$. On the other hand, first equality gives $M^{nt}=0$, that is $M$ is a nilpotent ideal. This implies that $E=M^{nt}E=0$. In this case, $R=A\propto (0)\cong A$ is an $n$-OAF-ring. Now, assume that $t=1$, that is, $(0)\propto E=J_1$ is an $n$-OA ideal. This gives $(M\propto E)^n=M^n\propto M^{n-1}E\subset (0)\propto E\subset M\propto E$. Which implies that $M^n=(0)$ and $M^{n-1}E\subsetneq E$. In this case, every ideal of $A$ is an $n$-OA ideal.\\
For the converse, assume that $A$ is an $n$-OAF ring with $E=0$. In this case, $A\cong A\propto (0)=R$ is an $n$-OAF-ring. Now, assume that $E\neq 0$, $M^n=0$, $M^{n-1}E\subseteq Mx$ for every $x\notin M^{n-1}E$ and $M^{n-1}E\subseteq aE$ for every nonzero $a\in M$. In this case, $R$ is a local ring with unique maximal ideal $M\propto E$. If $M^{n-1}E=0$, then by Lemma \ref{lemid} (2), we have $(M\propto E)^n=M^n\propto M^{n-1}E=(0)\propto (0)$. Then by \cite[Theorem 2]{B}, every ideal of $R$ is an $n$-OA ideal and hence $R$ is an $n$-OAF ring. Now, assume that $M^{n-1}E\neq (0)$. Then $(M\propto E)^{n+1}=M^{n+1}\propto M^n E=(0)\propto (0)$, that is, $M\propto E$ is a nilpotent ideal. Now, we will show that $(M\propto E)^n$ is a divided ideal of $R$. Let $(a,m)\notin (M\propto E)^n=(0)\propto M^{n-1}E$. This gives either $a\neq 0$ or $a=0$ and $m\notin M^{n-1}E$. Now, we have two cases. \textbf{Case 1:} Suppose that $a\neq 0$. Then we may assume that $a\in M$. Otherwise, $(a,m)$ is a unit in $R$ which implies that $(M\propto E)^{n}\subseteq R(a,m)=R$. Thus, $0\neq a\in M$. By the assumption, we have $M^{n-1}E\subseteq aE$. Let $(0,y)\in (M\propto E)^n=(0)\propto M^{n-1}E$. This gives $y\in M^{n-1}E\subseteq aE$, which implies that $y=ae$ for some $e\in E$. Then we can write $(0,y)=(0,ae)=(a,m)(0,e)\in (a,m)R$, that is $(M\propto E)^n\subseteq R(a,m)$. \textbf{Case 2:} Let $a=0$ and $m\notin M^{n-1}E$. Then by the assumption, we have $M^{n-1}E\subseteq Mm$. Let $(0,y)\in (M\propto E)^n$. Then we have $y\in M^{n-1}E\subseteq Mm$ which implies that $y=bm$ for some $b\in M$. This gives $(0,y)=(0,bm)=(a,m)(b,0)\in (a,m)R$, that is $(M\propto E)^n\subseteq R(a,m)$. Thus, $(M\propto E)^n$ is a divided ideal of $R$. Then by Proposition \ref{pdiv1}, $R$ is an $n$-OAF-ring.
\end{proof}

\begin{cor}\label{cid}
Let $E$ be a nonzero module over an integral domain $A$, $R=A\propto E$ and $n\in\mathbb{N}$. The following statements are equivalent. 
\begin{enumerate}
    \item $R$ is an $n$-OAF-ring.
    \item $A$ is a field.
    \item Every proper ideal of $R$ is an $n$-OA ideal.
\end{enumerate}
\end{cor}
\begin{proof}
$(1)\Rightarrow(2)$ Let $R$ be an $n$-OAF-ring and $E$ be a nonzero module over an integral domain $A$. Now, we have two cases. \textbf{Case 1:} Suppose that $R$ is not a local ring. Then by Lemma \ref{rm}, $R$ is a general ZPI-ring. Also by Theorem \ref{idealization}, $A$ is a general ZPI-ring and $ann(E)=P_1P_2\cdots P_m$ for some idempotent maximal ideals $P_i$ of $A$ (since $E\neq 0$). Thus, $A$ is a Dedekind domain and the only idempotent maximal ideal of $A$ is the zero ideal. Thus, $A$ is a field. \textbf{Case 2:} Suppose that $R$ is a local ring. In this case, $A$ is a local ring with unique maximal ideal $M$, and also, $M\propto E$ is the unique maximal ideal of $R$. As $E\neq 0$, by Theorem \ref{generalid}, $M$ is a nilpotent ideal. Since $M$ is a maximal ideal, we have $M=0$, that is $A$ is a field. \\
$(2)\Rightarrow (3)$ Since $A$ a field, $A$ is a local ring with the unique maximal ideal $M=(0)$. Then $R$ is a local ring with the unique maximal ideal $(0)\propto E$. Also, note that $[(0)\propto E]^n=(0)\propto (0)$, then by Lemma \ref{lmp} (2), every ideal of $R$ is an $n$-OA ideal.\\
$(3)\Rightarrow (1)$ It is straightforward.
\end{proof}

\begin{cor}\label{cid2}
Let $E$ be a nonzero module over a local ring $A$ with unique maximal $M$ such that $ME=0$, $R=A\propto E$ and $n\in\mathbb{N}$. The following are equivalent. 
\begin{enumerate}
    \item $R$ is an $n$-OAF-ring. 
    \item $M^n=0$.
    \item Every ideal of $R$ is an $n$-OA ideal.
\end{enumerate}
\end{cor}

\begin{proof}
$(1)\Rightarrow (2)$ Assume that $E$ is a nonzero module over a local ring $A$ with unique maximal ideal $M$. Then by Lemma \ref{lemid} (4), $R$ is a local ring with unique maximal ideal $M\propto E$. Since $R$ is an $n$-OAF-ring, by Theorem \ref{generalid}, we have $M^n=0$.\\
$(2)\Rightarrow(3)$ Assume that $M^n=0$. Since $ME=0$, we conclude that $(M\propto E)^n=M^n\propto M^{n-1}E=(0)\propto (0)$. Then by Lemma \ref{lmp} (2), every ideal of $R$ is an $n$-OA ideal.\\
$(3)\Rightarrow (1)$ It is straightforward.
\end{proof} 

To illustrate the application of Theorem \ref{generalid}, we present an example of an $n$-OAF ring that contains non $n$-OA ideals.

\begin{exam}
Let $k$ be a field and $A=k[X]/I$, where $I=(X^n)$ and $n\geq 2$. Then $A$ is a zero dimensional local ring with the unique maximal ideal $M=(x)$, where $x=X+I$. Let $R=A\propto A$. Then $R$ is a zero dimensional local ring with the unique maximal ideal $M\propto A$. Since $[M\propto A]^n=(0)\propto (x^{n-1})A\neq 0$, $R$ has non $n$-OA ideals by \cite[Theorem 2.22]{idealization}. It is easy to see that $A$ is a chain ring whose all ideals are $M^n=(0)\subset M^{n-1}=(x^{n-1})\subset (x^{n-2})\subset\cdots\subset M=(x)\subset A$. Let $E=A$ and $0\neq a\in M$. Then note that $M^{n-1}E=(x^{n-1})\subseteq aE$ since $(x^{n-1})$ is the minimal ideal of $A$. Let $y\notin M^{n-1}E=(x^{n-1})$. Since $A$ is a chain ring, $y=x^kb$ for some unit $b\in A$ and $k<n-1$. This gives $M^{n-1}E=(x^{n-1})\subseteq My=(x)(x^k)=(x^{k+1})$. Then by Theorem \ref{generalid} (2), $R$ is an $n$-OAF ring. 
\end{exam}

\end{document}